\documentclass[a4paper, 11pt]{amsart}

\usepackage{latexsym}
\usepackage{amssymb}
\usepackage{amsthm}
\usepackage{amscd}
\usepackage{amsmath}
\usepackage[all]{xy}
\usepackage{multicol}
\usepackage{float}
\usepackage{graphicx}
\usepackage{wrapfig}

\newtheorem{Theorem}{Theorem}[section]
\newtheorem{Lemma}[Theorem]{Lemma}
\newtheorem{Corollary}[Theorem]{Corollary}
\newtheorem{Proposition}[Theorem]{Proposition}
\newtheorem{Remark}[Theorem]{Remark}

\newtheorem{Example}[Theorem]{Example}

\newtheorem{Definition}[Theorem]{Definition}

\newtheorem{Question}[Theorem]{Question}

\newtheorem{Notation}[Theorem]{Notation}

  \makeatletter
    
    \@addtoreset{equation}{section}
  \makeatother

\def\fka{{\mathfrak a}}
\def\fkb{{\mathfrak b}}

\def\fkq{{\mathfrak q}}

\def\fkm{{\mathfrak m}}

\def\opn#1#2{\def#1{\operatorname{#2}}}

\opn\Spec{Spec}
\opn\Supp{Supp}
\opn\supp{supp}
\opn\Max{Max}
\opn\max{max}
\opn\Min{Min}
\opn\min{min}
\opn\Ass{Ass}
\opn\Assh{Assh}
\opn\Ann{Ann}
\opn\depth{depth}
\opn\rank{rank}

\opn\Mat{Mat}

\opn\Tot{Tot}

\opn\Sym{Sym}

\opn\ord{ord}

\opn\div{div}
\opn\Div{Div}
\opn\cl{cl}
\opn\Cl{Cl}

\opn\Ker{Ker}
\opn\Coker{Coker}
\opn\Im{Im}
\opn\Hom{Hom}
\opn\Tor{Tor}
\opn\Ext{Ext}
\opn\End{End}
\opn\Fitt{Fitt}
\opn\Aut{Aut}
\opn\id{id}

\opn\nat{nat}
\opn\pff{pf}
\opn\Pf{Pf}
\opn\GL{GL}
\opn\SL{SL}
\opn\G{G}
\opn\E{E}
\opn\H{H}
\opn\M{M}
\opn\mod{mod}
\opn\ord{ord}
\opn\det{det}
\opn\Soc{Soc}

\opn\chara{char}
\opn\length{\ell}
\opn\pd{pd}
\opn\rk{rk}
\opn\projdim{proj\,dim}
\opn\injdim{inj\,dim}
\opn\rank{rank}
\opn\depth{depth}
\opn\grade{grade}
\opn\height{ht}
\opn\embdim{emb\,dim}
\opn\codim{codim}

\renewcommand{\tilde}{\widetilde}
\renewcommand{\bar}{\overline}

\title{
Indecomposable 
integrally closed modules of arbitrary rank over a two-dimensional regular local ring
}
\author{Futoshi Hayasaka}
\address{Department of Environmental and Mathematical Sciences, Okayama University, 
3-1-1 Tsushimanaka, Kita-ku, Okayama, 700-8530, JAPAN}
\email{hayasaka@okayama-u.ac.jp}
\keywords{integral closure, indecomposable module, monomial ideal, regular local ring}
\subjclass[2020]{Primary 13B22; Secondary 13H05}

\begin{document}

\maketitle

\begin{abstract}
In this paper, we construct 
indecomposable integrally closed modules
of arbitrary rank over a two-dimensional
regular local ring. The modules are quite explicitly constructed from a given complete monomial ideal. 
We also give structural and numerical results on integrally closed modules. These are used in the proof of indecomposability of the modules. 
As a consequence, we have a large class of indecomposable integrally closed modules of arbitrary rank whose ideal is not necessarily simple. 
This extends the original result 
on the existence of indecomposable integrally closed modules and strengthens the non-triviality of the theory developed by Kodiyalam. 
\end{abstract}

\section{Introduction}

The theory of complete (integrally closed) ideals in two-dimensional regular local rings was developed by Zariski in \cite{Za} and further in \cite[Appendix 5]{ZS}. Zariski proved the product theorem and the unique factorization theorem. In a two-dimensional regular local ring, the product of any two complete ideals is again complete, and  
any non-zero complete ideal can be expressed uniquely 
(except for ordering) as a product of simple complete ideals. 
Here an ideal is said to be simple if it cannot be expressed as a product of two proper ideals. 
Since then, the classic theory has been deeply studied and extended by several authors, including 
Cutkosky (\cite{Cut1, Cut2, Cut}) and Lipman (\cite{Lip, Li}). 
See also \cite{Hu, HuSa, HuSw} for different treatments and other interesting results on this subject.  

On the other hand, Rees in \cite{Rees} introduced and studied the integral closure for a module. In \cite{Ko}, Kodiyalam initiated the study of integrally closed modules over a two-dimensional regular local ring in analogy with Zariski's theory on complete ideals. Among several interesting results, he established some relationships between integrally closed modules and certain Fitting ideals. In more detail, let $(R, \fkm)$ be a two-dimensional regular local ring with infinite residue field $R/\fkm$, 
and let $M$ be a finitely generated torsion-free $R$-module. We denote by $\bar M$ the integral closure of $M$ in the sense of 
Rees. 
We define the ideal $I(M)$ associated to $M$ as  
$$I(M)=\Fitt_0(M^{\ast \ast}/M), $$ 
where $M^{\ast \ast}$ is the double $R$-dual of $M$ which is a free $R$-module containing $M$. 
With this notation, Kodiyalam proved in \cite[Theorem 5.4]{Ko} that 
taking integral closure commutes with taking the ideal, that is, the equality 
$$I(\bar M)=\bar{I(M)}$$ 
holds true. In particular, the ideal $I(M)$ is complete if $M$ is integrally closed. This can be viewed as a generalization of the product theorem of Zariski. 
Indeed, if we consider a direct sum $M=\fka \oplus \fkb$ of two ideals 
$\fka$ and $\fkb$ in $R$, then we have the equality 
$$\bar \fka \cdot \bar \fkb =\bar{\fka \fkb} $$
since $\bar M=\bar \fka \oplus \bar \fkb$. 
This implies the product theorem of Zariski, and vice versa. Moreover, Kodiyalam proved the following (see \cite[Theorem 5.7]{Ko}). 

\begin{Theorem}[Kodiyalam]\label{kodi}
Let $(R, \fkm)$ be a two-dimensional regular local ring with the maximal ideal $\fkm$, 
infinite residue field $R/\fkm$. Let $M$ be a finitely generated torsion-free $R$-module. 
Suppose that $M$ has no free direct summand. Then 
$$\bar M \ \text{is an indecomposable integrally closed $R$-module},$$ 
if the ideal $\bar{I(M)}$ is simple. In particular, there exist indecomposable integrally closed $R$-modules of arbitrary rank $($with simple complete ideals$)$. 
\end{Theorem}

A motivation of this paper comes from the following question in \cite[Example 5.8]{Ko}: 
Does the converse to Theorem \ref{kodi}  hold in the sense that an indecomposable integrally closed $R$-module $M$ of rank bigger than $1$ have a simple complete ideal $I(M)$?

In \cite{Ha}, we show that the converse to Theorem \ref{kodi} does not hold by showing that 
there exist numerous counterexamples of rank 2. In fact, we prove the following  general result in \cite[Theorem 1.3]{Ha}: 
Let
\begin{equation}\label{completemono}
    I=\langle x^{a_i}y^{b_i} \mid 0 \leq i \leq r \rangle 
\end{equation} 
be an $\fkm$-primary complete monomial ideal 
with respect to a regular system of parameters $x, y$ for $R$ 
with $a_0>a_1> \dots > a_{r-1} > a_r=0, b_0=0<b_1< \dots < b_{r-1}< b_r$, and $a_0 \leq b_r$.

\begin{Theorem}$($\cite[Theorem 1.3]{Ha}$)$\label{prev}
Let $(R, \fkm)$ be a two-dimensional regular local ring with the maximal ideal $\fkm$, 
infinite residue field $R/\fkm$. Let $I$ be 
an $\fkm$-primary complete monomial ideal in $(\ref{completemono})$. Suppose that either 
\begin{enumerate}
\item $r \geq 3$, or
\item $r=2$, $xy \notin I$ $($and hence $x^2 \in I)$
\end{enumerate}
is satisfied. Then there exists a finitely generated torsion-free indecomposable 
integrally closed $R$-module $M$ of rank $2$ 
with $I(M)=I$.  
\end{Theorem}

This immediately implies that there exists a large class of indecomposable integrally closed $R$-modules $M$ of rank $2$ with non-simple complete ideal $I(M)$.
We refer to \cite{BiMo} for another type of examples of rank $2$ motivated by the question. 
Then it is natural to ask if there exist such modules of higher rank: 

\begin{Question}
{\rm 
Does there exist an indecomposable integrally closed $R$-module $M$ of arbitrary rank with non-simple ideal $I(M)$?
}
\end{Question}

The purpose of this paper is to show the existence of such modules. 
In fact, we prove a stronger result which shows that there exist a lot of such indecomposable integrally closed modules in arbitrary rank. 
The result can be summarized as follows:

\begin{Theorem}[Theorem \ref{6.1}]\label{main}
Let $(R, \fkm)$ be a two-dimensional regular local ring with the maximal ideal $\fkm$, 
infinite residue field $R/\fkm$. 
Let $I$ be 
an $\fkm$-primary complete monomial ideal in $(\ref{completemono})$. Let $e$ be an integer with $2 \leq e \leq r$ and $a_{r-e+2}=e-2$. 
Suppose that either 
\begin{enumerate}
\item $r \geq e+1$, or
\item $r=e$, $x^{r-1}y \notin I$ and $x^r \in I$
\end{enumerate}
is satisfied. Then there exists a finitely generated torsion-free indecomposable 
integrally closed $R$-module $M$ of rank $e$ 
with $I(M)=I$.  
\end{Theorem}

This immediately implies both Theorem \ref{prev} and the last assertion of Theorem \ref{kodi}. 
Indeed, if we take $e=2$, then we can readily get Theorem \ref{prev} (see Corollary \ref{6.2}). If we consider a simple complete monomial ideal, e.g. $I=\bar{\langle x^{r}, y^{r+1} \rangle}$, we can readily get 
the last assertion of Theorem \ref{kodi} (see Corollary \ref{6.4}).  
Moreover, we have the following as a direct consequence of Theorem \ref{main}. 

\begin{Corollary}
For any given integer $e \geq 2$, there exists a large class of indecomposable integrally closed $R$-modules $M$ of rank $e$ with non-simple ideal $I(M)$.
\end{Corollary}

Indeed, if we consider a non-simple complete monomial ideal $I$ and an integer $e$ satisfying $a_{r-e+2}=e-2$, e.g. $I=\prod_{i=1}^r \bar{\langle x, y^{i+1} \rangle}$ where $r \geq 2$, 
then Theorem \ref{main} shows that one can find an 
indecomposable integrally closed $R$-module $M$ of rank $e$ with $I(M)=I$. 
These modules are obtained quite explicitly 
from a given complete monomial ideal $I$ in (\ref{completemono}).

We explain the organization of this paper and the idea 
of our proof of Theorem \ref{main}. 
In section 2, we review basic facts on integrally closed $R$-modules from \cite{Ko} and on complete monomial ideals. 
In section 3, we introduce a certain module of rank $e$, denoted by $M(I;e)$, associated to a given monomial ideal $I$ 
and an integer $e$ (see Definition \ref{3.1}). 
The modules $M(I;e)$ play a central role in the proof of Theorem \ref{main}. 
In section 4, we prove two theorems on integrally closed $R$-modules which are of interest in its own right. 
One is a structure theorem (Theorem \ref{4.2}), and another is a numerical result (Theorem \ref{4.4}). 
These are used in the proof of  indecomposability of $M(I;e)$.   
In section 5, we investigate the indecomposability of $M(I;e)$. We first consider the case when $e \leq r-1$ and show that $M(I;e)$ is indecomposable if it is integrally closed (see Theorems \ref{3.7} and \ref{5.2}). 
Next we consider the case when $e=r$ and show the indecomposability of $M(I;e)$ under a certain additional condition (see Theorem \ref{5.4}). In section 6, we summarize our results and get Theorem \ref{main}. We give some consequences and examples. 

Throughout this paper, 
let $(R, \fkm)$ be a two-dimensional regular local ring with the maximal ideal 
$\fkm$, infinite residue field $R/\fkm$. 
For an ideal $\fka$ in $R$, 
the order of 
$\fka$ will be denoted by $\ord(\fka)=\max \{n \mid \fka \subset \fkm^{n} \}$. 
For an $R$-module $L$, the notations $\rank_{R}(L)$, $\mu_{R}(L)$ and $\ell_{R}(L)$ will 
denote respectively the rank, the minimal number of generators and the length of $L$. Let $L^{\ast}=\Hom_R(L, R)$ denote the $R$-dual of $L$. We will use both the term ``integrally closed'' and the classical one ``complete'' for ideals. 

\begin{Notation}
Let $A$ be an arbitrary Noetherian ring and let $A^n=At_1+\dots +At_n$ be a free $A$-module of rank $n>0$ 
with free basis $t_1, \dots , t_n$. 
\begin{itemize}
\item For a submodule $L=\langle f_1, \dots , f_m \rangle$ of $A^n$ generated by $f_1, \dots , f_m$, we denote the associated matrix
$$\tilde L:=(a_{ij}) \in \Mat_{n \times m}(A)$$
where $f_j=a_{1j}t_1+\dots +a_{nj}t_n$ for $j=1, \dots , m$. 
\item Conversely, for a matrix $\varphi=(b_{ij}) \in \Mat_{n \times m}(A)$, we denote the associated submodule of $A^n$
$$\langle \varphi \rangle=\langle g_1, \dots , g_m \rangle$$
where $g_j=b_{1j}t_1+\dots +b_{nj}t_n$ for $j=1, \dots , m$.
\item For two matrices $\varphi \in \Mat_{n \times m}(A)$ and $\psi \in \Mat_{n \times m'}(A)$ with the same number of rows, we define a relation $\sim$ as 
$$\varphi \sim \psi \Leftrightarrow \langle \varphi \rangle \cong \langle \psi \rangle \ \text{as $A$-modules}$$
\item For a matrix $\varphi=(b_{ij}) \in \Mat_{n \times m}(A)$, we denote the ideal in $A$ 
generated by all the $t$-minors of $\varphi$ 
$$I_{t}(\varphi)$$
\end{itemize}
\end{Notation}

\section{Preliminaries}

In this section, we review some basic facts from \cite{Ko} 
on integrally closed modules over $R$. 
For the details on a theory of complete ideals in $R$, see \cite{Hu, HuSw, ZS}. For the details on a theory of integral closure for modules, see \cite{Rees}.  

Let $M$ be a finitely generated torsion-free $R$-module, and 
let $F=M^{**}$ be the double $R$-dual of $M$. 
Then $F$ is $R$-free and it canonically contains $M$ with the quotient $F/M$ of finite length. 
Indeed, one can show that if $M$ is contained in a free $R$-module $G$ 
with the quotient $G/M$ of finite length, 
then there is a unique $R$-linear isomorphism $\varphi: F \to G$ that 
restricts to the identity map on $M$ (\cite[Proposition 2.1]{Ko}).
Thus, the two quotients $F/M \cong G/M$ are isomorphic as $R$-modules. 
In fact, $F/M$ is isomorphic to the $1$st local cohomology module $H^1_{\fkm}(M)$ of $M$ 
with respect to $\fkm$. Therefore, one can define ideals associated to $M$ as follows: 
\begin{align*}
I(M) &=\Fitt_0(F/M) \\
I_1(M) &= \Fitt_1(F/M)   
\end{align*}
where $\Fitt_i( \ast )$ denotes the $i$th Fitting ideal.

The notion of integral closure for a module was defined by Rees in \cite{Rees}. 
Let $A$ be an arbitrary Noetherian integral domain with the quotient field $K$. For an $A$-module $M$, let $M_K=M \otimes_R K$. A subring $S$ of $K$ containing $A$ is called birational overring of $A$. For 
such a ring $S$, let 
\begin{align*}
    MS&=\Im (M \otimes_A S \to M_K)
\end{align*}
denote an $S$-submodule of $M_K$ generated by $M$, 
which is isomorphic to the tensor product $M \otimes_{A} S$ modulo $S$-torsion. 
With this notation, the integral closure $\bar M$ is defined as follows: 
\begin{align*}
    \bar M&=\{f \in M_K \mid f \in MV \ \text{for every discrete valuation ring} \ 
V \ \text{with} \ A \subset V \subset K \}. 
\end{align*}
\noindent
The set $\bar M$ is an $A$-module containing $M$. 
$M$ is said to be integrally closed if $\bar M=M$. 
Since $R$ is a two-dimensional regular local ring, and, hence, it is normal, 
the integral closure $\bar M$ can be considered in the free module $F$ (see \cite[Theorem 1.8]{Rees}). 
Moreover, we have the following useful criteria for integral dependence of a module (see \cite{Rees} and also \cite[Theorem 3.2]{Ko}).  
\begin{align}
\bar M &= \{f \in F \mid f \in \Sym_R^1(F) \ \text{is integral over} \ \Sym_R(M) \}\label{eqcri} \\
&=\{f \in F \mid I(M+Rf) \subset \bar{I(M)}\}.\label{detcri}
\end{align}
The first one is called equational criterion, and the second one is called determinantal criterion. 
These criteria imply the following property. 
\begin{Proposition}$($\cite[Corollary 3.3]{Ko}\label{2.1}$)$
Let $M$ be a finitely generated torsion-free $R$-module.
Then $M$ is integrally closed if and only if so is $M_Q$
for every maximal ideal $Q$ of $R$. 
\end{Proposition}

As in the ideal case, the notion of
contracted modules plays an important 
role in a theory of integrally closed modules over $R$. 

\begin{Definition}$($\label{2.2}\cite[p.3554]{Ko}$)$
Let $S$ be a birational overring of $R$. 
Then a finitely generated 
torsion-free $R$-module $M$ 
is said to be contracted from $S$, 
if the equality 
$$MS \cap F=M$$ 
holds true as submodules of $FS$. 
\end{Definition}

Here are some basic properties of contracted modules. 

\begin{Proposition}\label{2.3}
Let $M$ be a finitely generated torsion-free $R$-module. 
\begin{enumerate}
    \item$($\cite[Proposition 2.2]{Ko}$)$ The following inequality holds true. 
    $$ \mu_R(M) \leq \ord(I(M))+\rank_R(M) $$
    \item$($\cite[Proposition 2.5 and Theorem 2.8]{Ko}$)$ The following are equivalent. 
    \begin{enumerate}
        \item The equality $\mu_R(M) = \ord(I(M))+\rank_R(M)$ holds true. 
        \item $M$ is contracted from $S=R[ \frac{\fkm}{x} ]$ for some $x \in \fkm \setminus \fkm^2$. 
        \item $M$ is contracted from $S=R[ \frac{\fkm}{x} ]$ for any $x \in \fkm \setminus \fkm^2$ such that $$\ell_R(R/[I(M)+\langle x \rangle])=\ord(I(M)). $$
    \end{enumerate}
When this is the case, $I(M)$ is also contracted from $S=R[ \frac{\fkm}{x} ]$ for some $x \in \fkm \setminus \fkm^2$.
\end{enumerate}
\end{Proposition}

\begin{Proposition}$($\cite[Proposition 4.3]{Ko}$)$\label{2.4}
Let $M$ be a finitely generated torsion-free $R$-module. If $M$ is integrally closed, then $M$ is contracted from $S=R[ \frac{\fkm}{x} ]$ for some $x \in \fkm \setminus \fkm^2$. 
\end{Proposition}

Consider a birational overring $S=R[\frac{\fkm}{x}]$ of $R$ 
where $x \in \fkm \setminus \fkm^2$. Then one can show that for any maximal ideal $Q$ of $R$, 
\begin{itemize}
\item $S_Q$ is a discrete valuation ring if $Q \nsupseteq \fkm S$. 
\item $S_Q$ is a two-dimensional regular local ring if $Q \supseteq \fkm S$. 
\end{itemize}
The two-dimensional regular local ring $S_Q$ for a maximal ideal $Q$ of $S$ containing $\fkm S$ is called a first quadratic 
transform of $R$. For an $\fkm$-primary ideal $I$ in $R$ with $\ord(I)=r$, we can write 
$IS=x^r[IS:_S x^r]. $
Then we define the ideal $I^{S}$ as 
$$I^S=IS:_S x^r$$ 
and call it a transform of $I$ in $S$. For a first quadratic transform $T:=S_Q$ of $R$, 
we also define a transform $I^{T}$ of $I$ in $T$ as 
$$I^T=I^S T. $$ 

Contracted modules have the following nice property: 

\begin{Proposition}\label{2.5}
Let $M$ be a finitely generated torsion-free $R$-module. 
Suppose that $M$ is contracted from $S=R[ \frac{\fkm}{x} ]$ for some 
$x \in \fkm \setminus \fkm^2$. 
Then the following are equivalent. 
\begin{enumerate}
\item $M$ is an integrally closed $R$-module. 
\item $MS$ is an integrally closed $S$-module. 
\end{enumerate}
When this is the case, 
for any first quadratic transform $T$ of $R$, $MT$ is an integrally closed $T$-module.  
\end{Proposition}

\begin{proof}
The assertion (1) $\Rightarrow$ (2) follows from 
\cite[the proof of Proposition 4.6]{Ko}. 
The converse (2) $\Rightarrow$ (1) follows from \cite[the proof of Theorem 5.2]{Ko}. The last assertion is \cite[Proposition 4.6]{Ko}.  
\end{proof}

Therefore, for an $\fkm$-primary integrally closed ideal $I$ in $R$, 
a transform $I^T$ of $I$ in a first quadratic transform $T=S_Q$ of $R$ is also integrally closed. 
Indeed, since $I$ is complete, $IS$ is integrally closed 
by Proposition \ref{2.5}, and, hence, 
$I^{S}$ and its 
localization $I^T$ are integrally closed. 

The ideal $I(M)$ of $M$ behaves well under transforms.

\begin{Proposition}$($\cite[Proposition 4.7]{Ko}$)$\label{2.6}
Let $M$ be a finitely generated torsion-free $R$-module and $T$ a first quadratic transform of $R$. Then the equality 
$$I(MT)=I(M)^T$$
holds true. 
\end{Proposition}

We recall some necessary facts on 
the integral closure of general monomial ideals (not necessarily in a polynomial ring over a field). 
We refer the readers to \cite{HbSw, HuSw, KiSt} for more results and the details on general monomial ideals. 

Let $\fka$ be an $\fkm$-primary monomial ideal in $R$ with respect to 
a regular system of parameters $x, y$. Suppose that $\fka$ is generated by a set of 
monomials $\{ x^{v_{i}}y^{w_{i}} \mid 1 \leq i \leq s \}$. Then, as in the usual monomial ideal case, 
one can define the Newton polyhedron ${\rm NP}(\fka)$ of $\fka$ as a convex hull of a set of exponent vectors of 
$\fka$ in $\mathbb R^2$. Namely, 
$${\rm NP}(\fka)=\left\{ (u_{1}, u_{2}) \mid (u_{1}, u_{2}) \geq 
\sum_{i=1}^{s} c_{i}(v_{i}, w_{i}) \ \text{for some} \ c_i \geq 0 \ \text{with} \ \sum_{i=1}^{s} c_{i}=1 \right\}. $$
Then the integral closure $\bar{\fka}$ of $\fka$ can be described as 
\begin{equation}\label{newton}
    \bar{\fka}=\langle 
x^{e_{1}}y^{e_{2}} \mid (e_{1}, e_{2}) \in \mathbb Z_{\geq 0}^2 \cap {\rm NP}(\fka) \rangle. 
\end{equation}
Thus, $\bar \fka$ is again a monomial ideal with respect to $x, y$. 

Let $\{(p_{i}, q_{i}) \mid 0 \leq i \leq t \}$ be a set of the vertices of ${\rm NP}(\fka)$ with 
$p_{0}>p_{1}>\dots >p_{t}=0$ and $q_{0}=0<q_{1}<\dots <q_{t}$. 
Then, by the description (\ref{newton}) of $\bar{\fka}$, it follows that 
$$\bar{\fka}=\bar{\langle x^{p_{i}}y^{q_{i}} \mid 0 \leq i \leq t \rangle}. $$
Moreover, one can see that 
$$\bar{\fka}=\prod_{i=1}^t \bar{ \langle x^{p_{i-1}-p_{i}}, y^{q_{i}-q_{i-1}} \rangle}. $$
Here we note that for a pair of positive integers $p', q'$ with ${\rm gcd}(p', q')=d$, 
$$\bar{\langle x^{p'}, y^{q'} \rangle}=\bar{\langle x^{p}, y^{q} \rangle}^{d}$$
where $p'=dp$ and $q'=dq$, and that for any $p, q>0$ with ${\rm gcd}(p, q)=1$, 
$$\bar{ \langle x^p, y^q \rangle} \ \text{is simple}.  $$
See \cite{GrKi} for more details on special simple ideals. 

Here is an example to illustrate this. 

\begin{Example}\label{exmono}
{\rm 
Let $$I=\langle x^8, x^6y, x^3y^2, x^2y^3, xy^4, y^8 \rangle$$
be a monomial ideal with respect to a regular system of parameters $x,y$ for $R$. Then the Newton polyhedron ${\rm NP}(I)$ is given in Figure \ref{fig1}, and $I$ is complete. The set of its vertices is 
$$\{(8,0), (3,2),(1,4),(0,8) \}$$
which is denoted by dots in Figure \ref{fig1}. Hence,  
\begin{align*}
    I&=\bar{\langle x^5, y^2 \rangle} \ \bar{\langle x^2, y^2} \rangle \langle x, y^4 \rangle \\
    &=\bar{\langle x^5, y^2 \rangle} \langle x, y \rangle^2 \langle x, y^4 \rangle. 
\end{align*}
}
\end{Example}

\begin{figure}[H]
\centering
\includegraphics[scale=0.7]{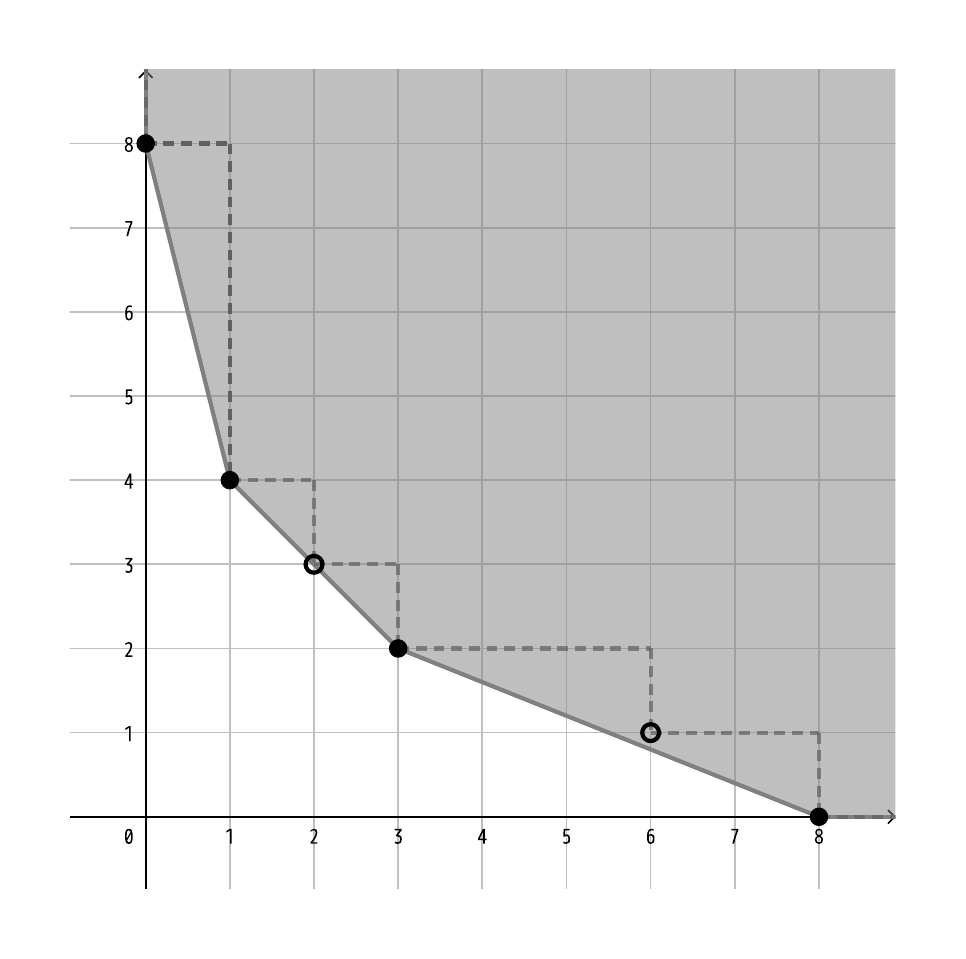}
\caption{}\label{fig1}
\end{figure}

Namely, for any $\fkm$-primary complete monomial ideal $\fka$ in $R$ with respect to $x, y$, 
every simple factor of $\bar{\fka}$ is a monomial ideal with the following special form: 
$$\bar{ \langle x^{p}, y^{q} \rangle } \ \text{where} \ {\rm gcd}(p, q)=1. $$
See \cite{BiTr, Qui} for more related results and the details 
on the factorization of usual complete monomial ideals. 

\section{Modules associated to a monomial ideal}\label{module}
In what follows, 
we fix a regular system of parameters $x, y$ for $R$, i.e. $\fkm=\langle x,y \rangle$, and consider a monomial ideal with respect to $x, y$
\begin{equation}\label{fixedmono}
    I=\langle x^{a_i}y^{b_i} \mid 0 \leq i \leq r \rangle  \tag{$\star$} 
\end{equation}
with $r \geq 2$ where 
\begin{align*}
    & a_0>a_1> \dots > a_{r-1} > a_r=0, \\
    & b_0=0<b_1< \dots < b_{r-1}< b_r, \ \text{and} \\
    & a_0 \leq b_r.  \\
\end{align*}

\begin{Lemma}\label{lem}
For the monomial ideal $I$ in $($\ref{fixedmono}$)$, we have \begin{enumerate}
    \item $\mu_R(I)=r+1$,  
    \item $\ord(I)=\min \{ a_i+b_i \mid 0 \leq i \leq r \}$, and  
    \item $\ell_R(R/[I+\langle x+y \rangle])=\ord(I)$.  
\end{enumerate}
Moreover, if $I$ is complete, then  
\begin{enumerate}
    \item[$(4)$] $\ord(I)=r$, and 
    \item[$(5)$] $a_{r-1}=1$. 
\end{enumerate}
\end{Lemma}

\begin{proof}
See \cite[Lemma 3.1]{Ha} for the assertions (1), (2), and (3). 
The assertion (4) follows from (1) by 
Propositions \ref{2.3} and \ref{2.4}. The last assertion (5) follows from the description (\ref{newton}) of complete monomial ideals.  
\end{proof}

\begin{Definition}\label{3.1}
Let $I$ be the monomial ideal in $($\ref{fixedmono}$)$, and let $e$ be an integer with $2 \leq e \leq r$. Let 
\begin{align*}
f_i&={}^t(x^{a_i'}y^{b_i}, 0, \dots , 0) \quad  \text{for} \quad 0 \leq i \leq r-e+1, \\    
g_i&={}^t(\dots, 0, \overset{\overset{i}{\smile}}{y}, \overset{\overset{i+1}{\smile}}{x}, 0, \dots)  \quad  \text{for} \quad 1 \leq i \leq e-1, \\
h_i&={}^t(\dots, 0, \overset{\overset{i+1}{\smile}}{y^{c_i}}, 0, \dots) \quad \text{for} \quad 
1 \leq i \leq e-1,  
\end{align*}
where  
\begin{align*}
    a_i'&=a_i-e+1 \quad \text{for} \quad 0 \leq i \leq r-e+1, \ \text{and}  \\ 
    c_i&=b_{r-e+1+i}-i \quad \text{for} \quad 1 \leq i \leq e-1.
\end{align*}
Then we define a module $M(I;e)$ associated to 
$I$ and $e$ as follows: 
\begin{equation*}
M(I;e)= \langle f_0, f_1, \dots , f_{r-e+1}, g_1, \dots , g_{e-1}, h_1, \dots , h_{e-1} \rangle \subset F:=R^e=\langle E_e \rangle    
\end{equation*}
where $E_e$ is an identity matrix of size $e \times e$. \end{Definition}

Note that the matrix of $M(I;e)$ is
\begin{equation*}
    \tilde{M(I;e)}= \left(\begin{array}{cccc|cccc|cccc}
x^{a_0'} & x^{a_1'}y^{b_1} & \cdots & x^{a_{r-e+1}'}y^{b_{r-e+1}} & y & & &  &  & & & \\
& & & & x & y & &  &y^{c_1}& & & \\  
& & & & & x & \ddots & & & y^{c_2} & & \\ 
& & & & &   & \ddots & y & & & \ddots & \\   
& & & & &  &  & x & & & & y^{c_{e-1}}  \\
\end{array}
\right)
\end{equation*}
of size $e \times (r+e)$, and 
\begin{itemize}
    \item $a_0'>a_1'> \dots >a_{r-e+1}' \geq 0$, 
    \item $c_{e-1} \geq c_{e-2} \geq \dots \geq c_1 \geq b_{r-e+1} \geq r-e+1$,
    \item $a_i' \geq r-i-e+1$ for $0 \leq i \leq r-e+1$.
\end{itemize}

Namely, the module $M(I;e)$ is a submodule of $F$ generated by column vectors of the above matrix $\tilde{M(I;e)}$, and the ideal $I(M(I;e))$ is an ideal generated by maximal minors of $\tilde{M(I;e)}$. 

\begin{Proposition}\label{3.2}
Let $M=M(I;e)$ where $2 \leq e \leq r$. Then  
$$I(M)=\langle x^{a_i}y^{b_i} \mid 0 \leq i \leq r-e+1 \rangle+
\langle x^{e-1-i}y^{c_i+i} \mid 1 \leq i \leq e-1 \rangle. $$
In particular, $I(M)$ is a monomial ideal containing $I$ and $\mu_R(I(M))=r+1.$
\end{Proposition}

To show this, we make the following lemma:  

\begin{Lemma}\label{3.3}
Let $1 \leq u_1 \leq \dots \leq u_k$ be a sequence of integers with $k \geq 1$. Let 
$$\psi=
\left(\begin{array}{ccccc|ccccc}
x & y & & &  & y^{u_1} & & && \\
 & x & y & &  && y^{u_2}& & \\  
 & & x & \ddots & & & &\ddots & & \\ 
 & &   & \ddots & y & & && \ddots & \\   
 & &  &  & x & & & & &y^{u_k}  \\
\end{array}
\right)
$$
be a matirix of size $k \times 2k$ over $R$, and let $\psi'$ be a matrix of size $k \times (2k-1)$ deleting the first column of $\psi$. Then we have the following: 
\begin{align*}
    & I_k(\psi)=\langle x^k \rangle + I_k(\psi'), \ \text{and}  \\
    & I_k(\psi')=\langle x^{k-i}y^{u_i+i-1} \mid 1 \leq i \leq k \rangle
\end{align*}
\end{Lemma}

\begin{proof}
We show the assertion by induction on $k$. When $k=1$, it is clear since 
$\psi=(x, y^{u_1})$ and 
$\psi'=(y^{u_1})$. 
Suppose that $k \geq 2$. 
Let $\psi_1$ be a matrix of size $(k-1) \times 2(k-1)$ deleting the 1st and $(k+1)$th columns and the 
1st row of $\psi$, and let $\psi_1'$ be a matrix deleting the 1st column of $\psi_1$. Then 
\begin{align*}
    I_k(\psi') &= yI_{k-1}(\psi_1')
    +y^{u_1}I_{k-1}(\psi_1)  \\
    &=yI_{k-1}(\psi_1')
    +y^{u_1}\left[ \langle x^{k-1} \rangle +I_{k-1}(\psi_1') \right]   \quad \text{by induction} \\
    &= yI_{k-1}(\psi_1')+\langle x^{k-1}y^{u_1} \rangle  \quad \text{since} \ u_1 \geq 1 \\
     &= y\langle x^{k-1-i}y^{u_{i+1}+i-1} \mid 1 \leq i \leq k-1 \rangle +\langle x^{k-1}y^{u_1} \rangle   \quad \text{by induction}  \\
     &=\langle x^{k-i}y^{u_i+i-1} \mid 1 \leq i \leq k \rangle, 
\end{align*}
and 
\begin{align*}
    I_k(\psi) &= xI_{k-1}(\psi_1)+I_k(\psi') \\
    &= x \left[ \langle x^{k-1} \rangle + I_{k-1}(\psi_1') \right]+I_k(\psi')   \quad \text{by induction}  \\ 
    &= \langle x^{k} \rangle +xI_{k-1}(\psi_1') + I_k(\psi')  \\
    &= \langle x^k \rangle + I_k(\psi')
\end{align*}
The last equality follows from 
\begin{align*}
    xI_{k-1}(\psi_1')&=\langle x^{k-i}y^{u_{i+1}+i-1} \mid 1 \leq i \leq k-1 \rangle \quad \text{by induction} \\
    &\subset I_k(\psi') \quad \text{since} \ u_{i+1}  \geq u_i
\end{align*} 
\end{proof}

\begin{proof}[Proof of Proposition \ref{3.2}]
Let $M=M(I;e)$, and let $\psi$ be a matrix of size $(e-1) \times 2(e-1)$
consisting of the last $(e-1)$ rows and the last $2(e-1)$ columns of $\tilde M$. Let $\psi'$ be a matrix of size $(e-1) \times (2(e-1)-1)$ deleting the 1st column of $\psi$. Then 
\begin{align*}
    I(M)&=I_e(\tilde M)\\
    &=\langle x^{a_i'}y^{b_i} \mid 0 \leq i \leq r-e+1 \rangle I_{e-1}(\psi) + y I_{e-1}(\psi') \\
    &=\langle x^{a_i'}y^{b_i} \mid 0 \leq i \leq r-e+1 \rangle \left[\langle x^{e-1} \rangle + I_{e-1}(\psi') \right]  + y I_{e-1}(\psi') \quad \text{by Lemma \ref{3.3}}\\
    &=\langle x^{a_i}y^{b_i} \mid 0 \leq i \leq r-e+1 \rangle + \langle x^{a_0'}, y\rangle I_{e-1}(\psi') \quad \text{since} \ b_i \geq 1 \ \text{for} \ 1 \leq i \leq r-e+1\\
    &=\langle x^{a_i}y^{b_i} \mid 0 \leq i \leq r-e+1 \rangle + \langle x^{a_0'}, y\rangle \langle x^{e-1-i}y^{c_i+i-1} \mid 1 \leq i \leq e-1 \rangle \\
    & \hspace{10.8cm} \text{by Lemma \ref{3.3}}\\
    &=\langle x^{a_i}y^{b_i} \mid 0 \leq i \leq r-e+1 \rangle \\
    &\hspace{3cm} +\langle x^{a_0-i}y^{c_i+i-1} \mid 1 \leq i \leq e-1 \rangle+ \langle x^{e-1-i}y^{c_i+i}  \mid 1 \leq i \leq e-1 \rangle. 
\end{align*}
Then the 1st and 3rd terms contain $I$ since 
$a_{r-e+1+i} \geq e-1-i$ and $c_i+i=b_{r-e+1+i}$.  
The 2nd term is contained in $I$ since
$a_0-i \geq a_i$ and $c_i+i-1=b_{r-e+1+i}-1 \geq b_{r-e+i} \geq b_i$. Hence, $I \subset I(M)$ and 
$$I(M)=\langle x^{a_i}y^{b_i} \mid 0 \leq i \leq r-e+1 \rangle+
\langle x^{e-1-i}y^{c_i+i} \mid 1 \leq i \leq e-1 \rangle.$$
The assertion $\mu_R(I(M))=r+1$ follows from Lemma \ref{lem}(1).
\end{proof}

Consequently, the ideal $I(M(I;e))$ is an 
$\fkm$-primary monomial ideal containing $I$ in (\ref{fixedmono}). Therefore, the quotient $F/M(I;e)$ has finite length, and 
for any $2 \leq e \leq r$, 
$$\rank_R(M(I;e))=e. $$ 

\begin{Proposition}\label{3.4}
Let $M=M(I;e)$ where $2 \leq e \leq r$. Then we have $\mu_R(M)=r+e$. 
\end{Proposition}

\begin{proof}
Suppose that 
\begin{equation}\label{eqsys}
\sum_{i=0}^{r-e+1}\alpha_i f_i+\sum_{i=1}^{e-1}(\beta_i g_i+\gamma_i h_i)={\bf 0}    
\end{equation}
where $\alpha_i, \beta_i, \gamma_i \in R$. 
We show that each $\alpha_i, \beta_i, \gamma_i \in \fkm$. 
By the $e$th row of (\ref{eqsys}), 
\begin{equation*}
    \beta_{e-1} x +\gamma_{e-1} y^{c_{e-1}} =0. 
\end{equation*}
Hence $\beta_{e-1} \in \langle y^{c_{e-1}} \rangle$ and  $\gamma_{e-1} \in \langle x \rangle$. 
Here we claim that

\medskip

\noindent
{\bf{Claim}}: for any $1 \leq i \leq e-2$, if $\beta_{i+1} \in \langle y^{c_{i+1}} \rangle$, then $\beta_i \in \langle y^{c_i} \rangle$ and $\gamma_i \in \langle x, y^{c_{i+1}-c_i+1} \rangle$. 

\begin{proof}
Write $\beta_{i+1}=\beta_{i+1}'y^{c_{i+1}}$ where $\beta_{i+1}' \in R$. 
Then, by the $(i+1)$th row of (\ref{eqsys}), 
$$\begin{array}{ccl}
    0&=&\beta_ix+\beta_{i+1}y+\gamma_iy^{c_i} \\
    &=&\beta_ix+\beta_{i+1}'y^{c_{i+1}+1}+\gamma_iy^{c_i} \\
    &=&\beta_ix+y^{c_i}(\beta_{i+1}'y^{c_{i+1}-c_i+1}+\gamma_i). 
\end{array}$$
Hence $\beta_i \in \langle y^{c_i} \rangle$ and $\gamma_i \in \langle x, y^{c_{i+1}-c_i+1} \rangle$. This proves the Claim. 
\end{proof}

\medskip

Therefore, $\beta_i \in \langle y^{c_i} \rangle$ and 
$\gamma_i \in \langle x, y \rangle$ for any $1 \leq i \leq e-1$. Write $\beta_1=\beta_1' y^{c_1}$ where $\beta_1' \in R$. By the 1st row of (\ref{eqsys}), 
\begin{align*}
    0&=\sum_{i=0}^{r-e+1} \alpha_i x^{a_i'}y^{b_i} +\beta_1 y \\
    &=\sum_{i=0}^{r-e+1} \alpha_i x^{a_i'}y^{b_i} +\beta_1' y^{c_1+1}. 
\end{align*}
Note that $c_1+1=b_{r-e+2}$. This implies that  
$\alpha_i \in \fkm$ for any $0 \leq i \leq r-e+1$. 
This completes the proof. 
\end{proof}

\begin{Corollary}\label{3.5}
Let $M=M(I;e)$ where $2 \leq e \leq r$. Consider the following conditions:
\begin{enumerate}
    \item $I$ is contracted from $S=R[\frac{\fkm}{x'}]$ for some $x' \in \fkm \setminus \fkm^2$. 
    \item $I(M)$ is contracted from $S=R[\frac{\fkm}{x'}]$ for some $x' \in \fkm \setminus \fkm^2$. 
    \item $M$ is contracted from $S=R[\frac{\fkm}{x'}]$ for some $x' \in \fkm \setminus \fkm^2$. 
\end{enumerate}
Then we have $(1)$ $\Rightarrow$ $(2)$ $\Leftrightarrow$ $(3)$. 
\end{Corollary}
\begin{proof}
We show (1) $\Rightarrow$ (2). 
Suppose that $I$ is contracted. Then by Proposition \ref{2.3}(2), 
$$\mu_R(I)=\ord(I)+1, $$ 
so $\ord(I)=r$ since $\mu_R(I)=r+1$ (Lemma \ref{lem}(1)). 
By Proposition \ref{3.2}, $I \subset I(M)$ and $\mu_R(I(M))=r+1$, so we have 
\begin{align*}
    r+1 &= \mu_R(I(M)) \\
    &\leq \ord(I(M))+1 \quad \text{by Proposition \ref{2.3}(1)}\\
    &\leq r+1. 
\end{align*}
Hence $\mu_R(I(M))=\ord(I(M))+1$. Thus, $I(M)$ is contracted by Proposition \ref{2.3}(2). 

Next, we show (2) $\Leftrightarrow$ (3). 
Suppose that $I(M)$ is contracted. Then  $$r+1=\mu_R(I(M))=\ord(I(M))+1, $$
so $\ord(I(M))=r$. By Proposition \ref{3.4}, $$\mu_R(M)=r+e=\ord(I(M))+\rank(M), $$ 
and, hence, 
$M$ is contracted by Proposition \ref{2.3}(2). 
Thus we have (2) $\Rightarrow$ (3). The converse follows from the last assertion of Proposition \ref{2.3}(2). This completes the proof.  
\end{proof}

The implication (2) $\Rightarrow$ (1) in Corollary \ref{3.5} 
does not hold in general when $e \geq 3$. 

\begin{Example}\label{3.6}
{\rm 
Let 
$$I=\langle x^{r+1-i}y^i \mid 0 \leq i \leq r-1 \rangle+\langle y^{r+1} \rangle$$ 
where $r \geq 3$. 
Since $\mu_R(I)=\ord(I)=r+1$, 
$I$ is not contracted by Proposition \ref{2.3}(2). 
Let $3 \leq e \leq r$, and let $M=M(I;e)$. 
Then, by Proposition \ref{3.2}, 
$$I(M)=\langle x^{r+1-i}y^i \mid 0 \leq i \leq r-e+1 \rangle +\langle x^{e-1-i}y^{r-e+1+i} \mid 1 \leq i \leq  e-2 \rangle +\langle y^{r+1} \rangle, $$
so $\mu_R(I(M))=r+1$ and $\ord(I(M))=r$. Hence, $I(M)$ is contracted.
}
\end{Example}

Here is the main result in this section. 

\begin{Theorem}\label{3.7}
Let $M=M(I;e)$ where $2 \leq e \leq r$. Then the following are equivalent:
\begin{enumerate}
    \item $I(M)$ is integrally closed. 
    \item $M$ is integrally closed. 
\end{enumerate}
\end{Theorem}
\begin{proof}
The assertion (2) $\Rightarrow$ (1) follows from \cite[Theorem 5.4]{Ko}. 
We show (1) $\Rightarrow$ (2). 
Suppose that $I(M)$ is integrally closed. Then 
$I(M)$ is 
contracted 
by Proposition \ref{2.4}, and, hence, $M$ is contracted by Corollary \ref{3.5}. 
Let $S=R[\frac{\fkm}{x+y}]$. Then 
$$\ell_R(R/[I(M)+\langle x+y \rangle])=\ord(I(M))$$ 
by Lemma \ref{lem}(3), so 
$M$ is contracted from $S$ by Proposition \ref{2.3}(2). By Proposition \ref{2.5}, it is enough to show that $MS$ is integrally closed. 
Let $z=\frac{x}{x+y} \in S$. Then, in $S$, we can write 
\begin{equation}\label{xy}
    x=z(x+y) \quad \text{and} \quad y=(1-z)(x+y). 
\end{equation} 
We substitute 
the right hand side in (\ref{xy}) for $x, y$ in the matrix $\tilde{MS}$. Then, by considering the $S$-linear map $S^e \to S^e$ defined by $(x+y)E_e$, we have
\begin{equation*}
    \tilde{MS} \sim  
    \left( \begin{array}{ccc|cccc|cccc} 
    \ast&\cdots&\ast&1-z& &  &  &  &  & & \\ 
    & & &z&1-z& & & \ast& & & \\ 
    & & &  &z&\ddots& & &\ast& & \\ 
    & & & & &\ddots&1-z& & &  \ddots &\\ 
    & & & & & &z& & && \ast  
    \end{array} \right). 
\end{equation*}
By Proposition \ref{2.1}, it is enough to show that 
$MS_Q$ is integrally closed for any maximal ideal $Q$ of $S$.

Let $Q$ be a maximal ideal of $S$. If $Q \nsupseteq \fkm S$, then $S_Q$ is a discrete valuation ring. Therefore, 
$MS_Q$ is integrally closed because of 
the fact that any submodule of finitely generated free module over a discrete valuation ring 
is integrally closed.

Suppose 
that $Q \supseteq \fkm S$. 
If $z \notin Q$, then $z$ is a unit of $S_Q$, and, by elementary matrix operations over $S_Q$, we have
\begin{equation*}
    \tilde{MS_Q} \sim 
    \left( \begin{array}{ccc|cccc}
        \ast &\cdots&\ast&&&&  \\
         &&&1&&&  \\
         &&&&1&&  \\
         &&&&&\ddots&  \\
         &&&&&&1  \\
    \end{array} \right). 
\end{equation*}
Therefore, $MS_Q \cong J \oplus S_Q^{e-1}$ for some ideal $J$ of $S_Q$. 
We may assume that $J$ is $QS_Q$-primary. 
By Proposition \ref{2.6}, 
$$J=I(MS_Q)=I(M)^{S_Q}, $$ 
and, hence, 
$J$ is complete. Thus, $MS_Q$ is integrally closed. 
Suppose that $z \in Q$, so that $1-z \notin Q$. 
Since $1-z$ is a unit of $S_Q$, we have 
\begin{equation*}
    \tilde{MS_Q} \sim 
    \left( \begin{array}{ccc|cccc}
         &&&1&&&  \\
         &&&&1&&  \\
         &&&&&\ddots&  \\
         &&&&&&1  \\
         \ast &\cdots&\ast&&&&  \\
    \end{array} \right). 
\end{equation*}
Therefore, by the same arguments as above, 
$MS_Q \cong S_Q^{e-1} \oplus J$ for some $QS_Q$-primary complete ideal $J$ of $S_Q$. Thus, $MS_Q$ is integrally closed. This proves that $M$ is integrally closed. 
\end{proof}

\begin{Remark}\label{3.9}
{\rm 
If the monomial ideal $I$ in (\ref{fixedmono}) is complete and 
an integer $e$ satisfies $a_{r-e+2}=e-2$, then $I=I(M)$ by Proposition \ref{3.2}. Hence, 
the associated module $M=M(I;e)$ is integrally closed by Theorem \ref{3.7}. Note that the condition $a_{r-e+2}=e-2$ is always satisfied when $e=2$. 
Also, if $I$ is complete, it is satisfied when $e=3$ (Lemma \ref{lem}(5)). Therefore, if $I$ is complete, then the associated modules $M(I;2)$ and $M(I;3)$ are always integrally closed. }
\end{Remark}

However, even if the monomial ideal $I$ in (\ref{fixedmono}) is complete, the associated module $M(I;e)$ is not necessarily 
integrally closed when $e \geq 4$. 

\begin{Example}
{\rm 
Let $$I=\langle x, y^{r-1} \rangle \bar{\langle x^{2r-3}, y^{r-1} \rangle }$$ where $r \geq 4$. Then $I$ is complete. 
Let $4 \leq e \leq r$, and let $M=M(I;e)$. 
Then 
$$I(M)=I+\langle x^{e-1-i}y^{r-e+1+i} \mid 1 \leq i \leq e-3 \rangle $$
by Proposition \ref{3.2}, and hence, 
$$\bar{I(M)}=\langle x, y^{r-1} \rangle \langle x,y \rangle^{e-3} \bar{\langle x^{2r-e}, y^{r-e+2} \rangle}. $$
Since $x^{2(r-2)}y \in \bar{I(M)} \setminus I(M)$, $I(M)$ is not complete, and, hence, $M$ is not integrally closed by Theorem \ref{3.7}. 
}
\end{Example}

\section{Two results on integrally closed modules}\label{two}
In this section, we will prove two results on integrally closed $R$-modules, which are of interest in its own right. 
These will play key role in our proof of indecomposablity of the module 
$M(I;e)$ defined in section \ref{module}. 

We first recall the following useful lemma
from Gaffney \cite[Proposition 1.5]{Ga} in a particular case.

\begin{Lemma}\label{4.1}
Let $A$ be an arbitrary Noetherian ring, and let $J$ be its Jacobson radical. Let $\fkb \subset \fka$ be ideals in $A$. If the equality $$\bar{J \bar{\fka}+\fkb}=\bar{\fka}$$
holds, then we have $\bar{\fka}=\bar{\fkb}$. 
\end{Lemma}
\begin{proof}
Since $J\bar{\fka}+\fkb \subset \bar{J\bar{\fka}+\fkb}=\bar{\fka}$ is a reduction, there exists an integer $\ell \geq 0$ such that 
$$\begin{array}{ccc}
     \bar{\fka}^{\ell+1}&=&(J\bar{\fka}+\fkb)\bar{\fka}^{\ell}  \\
     &=& J \bar{\fka}^{\ell+1}+\fkb \bar{\fka}^{\ell}.  
\end{array}$$
Hence, $\bar{\fka}^{\ell+1}=\fkb \bar{\fka}^{\ell}$ by Nakayama's Lemma. Therefore, $\fkb \subset \bar{\fka}$ is a reduction, so $\bar{\fka}=\bar{\fkb}$. 
\end{proof}

Let us recall the notion of reduction of modules. Let $A$ be an arbitrary Noetherian integral domain, and let $M$ be a finitely generated torsion-free $A$-module. Then an $A$-module $N$ 
is said to be 
a reduction of $M$, if $N \subset M \subset \bar{N}$. 
Note that if $A=R$ our base ring, then by the determinantal criterion (\ref{detcri}), 
$$N \subset M \ \text{is a reduction} \ \Leftrightarrow \ I(N) \subset I(M) \ \text{is a reduction.} $$ 

Here is our first main result in this section. 

\begin{Theorem}\label{4.2}
Let $M$ be a finitely generated torsion-free, integrally closed $R$-module without a free direct summand, and let $e=\rank_R (M) \geq 2$. 
Assume that
\begin{enumerate}
    \item $\ord(I(M)) \geq e+1$, 
    \item $\bar{I_1(M)}=\fkm^{e-1}$, and 
    \item $M$ is not indecomposable. 
\end{enumerate}
Then there exist an integer $1 \leq \ell \leq e$ and an indecomposable integrally closed $R$-module $L$ such that 
$$M \cong \fkm^{\oplus \ell} \oplus L. $$
In particular, $M$ has the maximal ideal $\fkm$ as a direct summand. 
\end{Theorem}
\begin{proof}
By the assumption (3), we can write $M \cong L_1 \oplus L_2$
for some non-zero $R$-modules $L_1$ and $L_2$. 
Since $M$ is integrally closed, so are $L_1$ and $L_2$. 
Note that $L_1 \subset \fkm L_1^{\ast \ast}$ and $L_2 \subset \fkm L_2^{\ast \ast}$ because $L_1$ and $L_2$ have no free direct summand. 
Let $e_1=\rank_R(L_1)$ and $e_2=\rank_R(L_2)$. 
Then 
$$I(M)=I(L_1)I(L_2) $$ 
and hence, $\ord(I(L_1)) \geq e_1+1$ or $\ord(I(L_2)) \geq e_2+1$ by the assumption (1). 
We may assume that $\ord(I(L_2)) \geq e_2+1$. Then we claim that 
\medskip

\noindent{\bf Claim}: $\ord(I(L_1))=e_1$. 
\begin{proof}
Since $L_1 \subset \fkm L_1^{\ast \ast}$, 
we have $\ord(I(L_1)) \geq e_1$ and 
$\ord(I_1(L_1)) \geq e_1-1$. Similarly, $\ord(I_1(L_2)) \geq e_2-1$. 
Suppose that $\ord(I(L_1)) \geq e_1+1$. Then $\bar{I_1(M)} \subset \bar{\fkm^e}=\fkm^e$  since 
$$I_1(M)=I(L_1)I_1(L_2)+I_1(L_1)I(L_2). $$
This contradicts to 
the assumption (2). 
\end{proof}

\noindent
Hence, we have $I_1(M) \subset I(L_1)I_1(L_2)+\fkm^e \subset \fkm^{e-1}. $ 
By taking the integral closures, 
$$\bar{I_1(M)}=\bar{I(L_1)I_1(L_2)+\fkm^e}=\bar{\fkm^{e-1}}. $$
Note that $\fkm^e=\fkm \bar{\fkm}^{e-1}$. 
By Lemma \ref{4.1}, $$\bar{I(L_1)I_1(L_2)}=\bar{\fkm^{e-1}}=\fkm^{e-1}. $$ 
By Zariski's product theorem, 
$$\bar{I(L_1)} \cdot \bar{I_1(L_2)}=\fkm^{e-1}. $$ 
Moreover, by Zariski's unique factorization theorem, we have
$$ \bar{I(L_1)}=\fkm^{e_1} \quad \text{and} \quad \bar{I_1(L_2)}=\fkm^{e_2-1}. $$
The first equality implies that 
$L_1 \subset \fkm L_1^{\ast \ast}$ is a reduction because 
$I(L_1) \subset I(\fkm L_1^{\ast \ast})=\fkm^{e_1}$ is a reduction. 
Hence, 
$$L_1=\bar{L_1}=\bar{\fkm L_1^{\ast \ast}}=\fkm L_1^{\ast \ast} \cong \fkm^{\oplus e_1}. $$
Therefore, we have $M \cong \fkm^{\oplus e_1} \oplus L_2$. 
If $L_2$ is not indecomposable, by proceeding the same steps, we finally get the desired form $M \cong \fkm^{\ell} \oplus L$ for some integer $1 \leq \ell \leq e$ and some indecomposable integrally closed $R$-module $L$. This completes the proof. 
\end{proof}

\begin{Remark}\label{4.3}
{\rm 
The product and unique factorization theorems of Zariski have been generalized to other classes in two-dimensional normal local domains (see the works of Cutkosky in \cite{Cut1, Cut2, Cut} and Lipman in \cite{Lip, Li}). The above proof shows that over such rings, 
Theorem \ref{4.2} holds true. It would be interesting to construct modules of arbitrary rank satisfying the assumptions in Theorem \ref{4.2} over such rings}.  
\end{Remark}

Another main result in this section is as follows:

\begin{Theorem}\label{4.4}
Let $M$ be a finitely generated torsion-free, integrally closed $R$-module with $e=\rank_R(M) \geq 1$. Then we have the following inequality: 
$$\ell_R(R/I(M))-\ell_R(M^{\ast \ast}/M) \geq \frac{e(e-1)}{2}. $$
\end{Theorem}

Before the proof, we recall some facts on reductions and the Buchsbaum-Rim multiplicity for modules. 
Let $(A, \fkm)$ be a Noetherian local domain of $\dim A =d$, and let 
$M$ be a finitely generated torsion-free $A$-module. Let $N$ be 
a reduction of $M$. Then the inequality 
$$\mu_R(N) \geq \rank_R(M) +d-1$$
holds true. $N$ is said to be a minimal reduction of $M$ if $N$ itself
has no proper reduction. As in the ideal case, one can show that 
minimal generating set of $N$ forms a part of minimal generating 
set of $M$. Therefoere, the equality 
$$\mu_R(N)=\rank_R(M)+d-1$$ 
implies that $N$ is a minimal reduction of $M$. If we further assume that the residue field $A/\fkm$ is infinite, then the converse holds true:

\begin{Theorem}$($\cite[Lemma 2.2]{Rees}$)$\label{4.5}
Let $(A, \fkm)$ be a Noetherian local domain of $\dim A=d$ 
with infinite residue field $A/\fkm$, and let $M$ be a finitely generated torsion-free $A$-module. 
Then $M$ has a (minimal) reduction which is generated by $\rank_R(M)+d-1$ elements.  
\end{Theorem}

We make the following lemma which 
follows from the proof of \cite[Proposition 4.1]{KM} in a particular case. 
We provide a proof for the reader's convenience. 

\begin{Lemma}\label{4.6}
Let $\fka$ be an $\fkm$-primary ideal in $R$ with $\mu:=\mu_R(\fka) \geq 3$. Let $\fkq$ be a minimal reduction of $\fka$. Then there exists an $R$-submodule $N$ of $R^{\mu-2}$ such that 
\begin{enumerate}
    \item $N \subset \fkm R^{\mu-2}$, 
    \item $\mu_R(N)=\mu-1$, and 
    \item $R^{\mu-2}/N \cong \fka/\fkq$.
\end{enumerate} 
\end{Lemma}
\begin{proof}
Since $R/\fkm$ is infinite, $\fkq$ 
is a parameter ideal in $R$. Let $\fkq=\langle x_1, x_2 \rangle$. Extending the minimal generating set to the one of $\fka$, let 
$\fka=\langle x_1, x_2, x_3, \dots , x_{\mu} \rangle$. Consider the following 
commutative diagram: 
\[
  \begin{CD}
     \mathbb K : \ \ \ \  0 @>>>  R  @>{d_2}>>  R^2  @>{d_1}>>  \fkq  @>>>  0 \\
    @.     @VV{\varphi_2}V  @VV{\varphi_1}V  @VV{i}V   @. \\
     \mathbb F : \ \ \ \  0 @>>>  R^{\mu-1} @>{\delta_2}>>  R^{\mu} @>{\delta_1}>>  \fka @>>>  0
  \end{CD}
\]

\noindent
where $\mathbb K$ is the Koszul complex for $x_1, x_2$, $\mathbb F$ is a minimal free resolution of $\fka$ with an inclusion map $i$ and 
$\varphi_1=\left( \begin{array}{c} E_2 \\ O \end{array} \right)$. Taking the mapping cone, we have the following exact sequence:

\[
  \begin{CD}
     \mathbb M : \ \ \ \  
     0 @>>>  R  @>{{\footnotesize \left( \begin{array}{c}
         -d_2  \\
          \varphi_2 
     \end{array}\right)}}>>  \begin{array}{c}
          R^2  \\
          \oplus \\
          R^{\mu-1}
     \end{array}  @>{{\footnotesize \left( \begin{array}{cc}
         -d_1 & O \\
         \varphi_1 & \delta_2
     \end{array}\right)} }>>  \begin{array}{c}
         \fkq  \\
          \oplus \\
          R^{\mu} 
     \end{array}  @>{{\footnotesize \left( \begin{array}{cc} 
     i  & \delta_1 \end{array} \right)}}>>  \fka @>>>  0. 
  \end{CD}
\]

\noindent
This implies the following exact sequence: 

\[
  \begin{CD}
     0 @>>>  R  @>{{\footnotesize \left( \begin{array}{c}
         -d_2  \\
          \varphi_2 
     \end{array}\right)}}>>  \begin{array}{c}
          R^2  \\
          \oplus \\
          R^{\mu-1}
     \end{array}  @>{{\footnotesize \left( \begin{array}{cc}
         \varphi_1 & \delta_2
     \end{array}\right)} }>> 
          R^{\mu} 
      @>{\delta_1}>>  \fka/\fkq @>>>  0. 
  \end{CD}
\]

\noindent
Moreover, by splitting off free direct summands, we have the following exact sequence: 

\[
  \begin{CD}
     0 @>>>  R  @>{\varphi_2}>>  R^{\mu-1}
     @>{\tilde{\delta_2}}>> 
          R^{\mu-2} 
      @>{\tilde{\delta_1}}>>  \fka/\fkq @>>>  0 
  \end{CD}
\]

\noindent
where $\delta_1=(d_1, \tilde{\delta_1})$ and 
$(\varphi_1, \delta_2)=\left( \begin{array}{cc}
    E_2 & \ast \\
    O & \tilde{\delta_2}
\end{array}\right)$. 
Let $N=\Im \tilde{\delta_2}$. Then $N$ satisfies that 
$N \subset \fkm R^{\mu-2}$, $\mu_R(N)=\mu-1$, and 
$R^{\mu-2}/N \cong \fka/\fkq$.
This completes the proof. 
\end{proof}

Let $(A, \fkm)$ be a Noetherian local ring of $\dim A=d>0$, and let $C$ be an $A$-module of finite length. Consider a minimal free presentation of $C$:
\[
  \begin{CD}
     A^m  @>{\varphi}>>  A^{n}
     @>>>  C  @>>>  0.   
  \end{CD}
\]
The map $\varphi$ induces the natural homomorphism $\tilde{\varphi}: \Sym_A(A^m) \to \Sym_A(A^n)$ of symmetric algebras. 
Then one can consider the length function
$$ \lambda(p)=\ell_A([\Coker \tilde{\varphi}]_p). $$ 
Buchsbaum and Rim proved in \cite{BR} that this function is eventually a polynomial 
of degree $d+n-1$, and defined the multiplicity of $C$ as 
$$e(C)=\lim_{p \to \infty}\frac{\lambda(p)}{p^{d+n-1}}(d-n+1)!$$
This is a positive integer, and it does not depend on a choice of minimal free presentations of $C$. 
This invariant $e(C)$ is now called the Buchsbaum-Rim multiplicity of $C$. 

Here is a fundamental formula for the Buchsbaum-Rim multiplicity of a module over a Cohen-Macaulay local ring. 

\begin{Theorem}\label{4.7}
Let $(A, \fkm)$ be a Cohen-Macaulay local ring of $\dim A=d>0$, and 
let $M$ be a submodule of $F:=A^n$ with finite colength. If $N$ is a (minimal) reduction of $M$ with $\mu_A(N)=d+n-1$, then we have 
$$e(F/M)=e(F/N)=\ell_A(F/N)=\ell_A(A/\Fitt_0(F/N)).$$
\end{Theorem}

For the first equality, see \cite[Proposition 3.8]{Ko} for instance. 
The second equality follows from \cite[Corollary 4.5]{BR}, and the third one from \cite[2.10]{BV}. See also \cite[Theorem 1.3 (2)]{HH}. 

We require the following interesting formula, which is given by Kodiyalam and Mohan in \cite{KM}, for integrally closed modules over a two-dimensional regular local ring. 
Let $e(M)$ denote the Buchsbaum-Rim multiplicity $e(M^{\ast \ast}/M)$ of an $R$-module $M$. 

\begin{Theorem}$($\cite[Corollary 4.3]{KM}$)$\label{4.8}
Let $M$ be a finitely generated torsion-free $R$-module. If $M$ is integrally closed, then we have 
$$\ell_R(R/I(M))-\ell_R(M^{\ast \ast}/M)=e(I(M))-e(M).$$
\end{Theorem}

We are now ready to prove Theorem \ref{4.4}. 

\begin{proof}[Proof of Theorem \ref{4.4}]
Let $M$ be an integrally closed $R$-module with $e=\rank_R(M) \geq 1$. 
The case where $e=1$ is clear. Suppose that $e \geq 2$. We may assume that $M$ has no free direct summand. 
Let $N$ be a minimal reduction of $M$. 
Then $\mu_R(N)=e+1$ and hence, $\mu_R(I(N))=e+1$. By Theorem \ref{4.8}, $$\ell_R(R/I(M))-\ell_R(M^{\ast \ast}/M)=e(I(M))-e(M). $$
Since $N$ is a minimal reduction of $M$, 
$$e(M)=\ell_R(R/I(N))$$
by Theorem \ref{4.7}, and $I(N)$ is
a reduction of $I(M)$ by the determinantal criterion (\ref{detcri}). Hence,  
$$e(I(M))=e(I(N)). $$
Therefore, 
$$\ell_R(R/I(M))-\ell_R(M^{\ast \ast}/M)=e(I(N))-\ell_R(R/I(N)). $$
Here we claim the following:  

\medskip

\noindent
{\bf Claim}: For any $\fkm$-primary ideal
$\fka$ in $R$ with 
$\mu:=\mu_R(\fka) \geq 3$, 
we have the inequality
\begin{equation*}
    e(\fka)-\ell_R(R/\fka) 
    \geq \binom{\mu-1}{2}.
\end{equation*}

\begin{proof}
Let $\fkq$ be a minimal reduction of $\fka$. Then $e(\fka)=\ell_R(R/\fkq)$. Hence 
\begin{align*}
    e(\fka)-\ell_R(R/\fka)&=\ell_R(R/\fkq)-\ell_R(R/\fka)\\
    &=\ell_R(\fka/\fkq). 
\end{align*}
By Lemma \ref{4.6}, 
there exists an $R$-submodule $N$ of $R^{\mu-2}$ such that
$N \subset \fkm R^{\mu-2}$, $\mu_R(N)=\mu-1$ and $R^{\mu-2}/N \cong \fka/\fkq$. Hence, 
\begin{align*}
    \ell_R(\fka/\fkq)&=\ell_R(R^{\mu-2}/N)\\
    &=e(N) \quad \text{by Theorem \ref{4.7}}\\
    &\geq e(\fkm R^{\mu-2}). 
\end{align*}
Let $N'$ be a submodule of $\fkm R^{\mu-2}$ generated by a matrix
$$\left( 
\begin{array}{ccccc}
    x & y& & & \\
     & x& y& & \\
     & & \ddots&\ddots & \\
     & & &x &y \\
\end{array}\right)$$
of size $(\mu-2) \times (\mu-1)$ over $R$. Then, since $I(N')=\fkm^{\mu-2}=I(\fkm R^{\mu-2})$, $N'$ is a minimal reduction of $\fkm R^{\mu-2}$. Therefore, by Theorem \ref{4.7}, 
$$e(\fkm R^{\mu-2})=\ell_R(R/I(N'))=\binom{\mu-1}{2}. $$
This completes the proof of Claim.
\end{proof} 
\noindent 
By applying the Claim for $I(N)$, we have 
\begin{align*}
\ell_R(R/I(M))-\ell_R(M^{\ast \ast}/M)&=e(I(N))-\ell_R(R/I(N)) \\
&\geq \binom{(e+1)-1}{2} \\
&=\binom{e}{2}. 
\end{align*}
\end{proof}

\section{Indecomposability}
In this section, we will investigate the indecomposablity of modules
$M(I;e)$ defined in section \ref{module} using results in section \ref{two}. We begin with the following: 

\begin{Proposition}\label{5.1}
Let $M=M(I;e)$ where $2 \leq e \leq r$. If $I(M)$ is complete, then we have the equality
$$\ell_R(R/I(M))-\ell_R(F/M)=\binom{e}{2}. $$
\end{Proposition}

\begin{proof}
Suppose that $I(M)$ is complete. 
Then $M$ is integrally closed by Theorem \ref{3.7}, and, hence, 
$$\ell_R(R/I(M))-\ell_R(F/M) \geq \binom{e}{2} $$
by Theorem \ref{4.4}. 
To show the converse inequality, we consider the following surjective $R$-linear map: 
$$\begin{CD}
   \varphi: \ F @>>> \fkm^{e-1}/I(M) 
\end{CD}
$$
defined by a matrix 
$$\varphi=(x^{e-1}, -x^{e-2}y, \cdots, (-1)^{i-1}x^{e-i}y^{i-1}, \cdots, (-1)^{e-1}y^{e-1}). $$ 
Then, by Proposition \ref{3.3}, it is easy to see that $\varphi(f_i)=0$ for all $0 \leq i \leq r-e+1$ and also 
$\varphi(g_i)=\varphi(h_i)=0$ for all $1 \leq i \leq e-1$. Hence, 
$M \subset \Ker \varphi$ and we have the surjective $R$-linear map 
$$\begin{CD}
   F/M @>>> F/\Ker \varphi \cong \fkm^{e-1}/I(M). 
\end{CD}$$
Therefore,  
\begin{align*}
\ell_R(F/M) &\geq \ell_R(\fkm^{e-1}/I(M)) \\
&= \ell_R(R/I(M))-\ell_R(R/\fkm^{e-1}) \\
&= \ell_R(R/I(M))-\binom{e}{2}. 
\end{align*}
This completes the proof. 
\end{proof}

We first consider the case where $e \leq r-1$. 

\begin{Theorem}\label{5.2}
Let $M=M(I;e)$ where $2 \leq e \leq r-1$. If $I(M)$ is complete, then $M$ is an indecomposable integrally closed $R$-module. 
\end{Theorem}

\begin{proof}
By Theorem \ref{3.7}, $M$ is integrally closed. Suppose that $M$ is not indecomposable $R$-module. 
By assumption, $e+1 \leq r =\ord(I(M))$. It is clear that $I_1(M)=\fkm^{e-1}$, so $\bar{I_1(M)}=\fkm^{e-1}$. Therefore, 
by Theorem \ref{4.2}, $M$ has the maximal ideal $\fkm$ as a direct summand. Hence, we can write $M \cong \fkm \oplus N$ for some integrally closed $R$-module $N$. Then $I(M)=\fkm I(N)$ and  
\begin{align*}
    \ell_R(R/I(M))-\ell_R(F/M)&=\ell_R(R/\fkm I(N))-(1+\ell_R(N^{\ast \ast}/N)) \\
    &=\ell_R(R/I(N))+\mu_R(I(N))-1-\ell_R(N^{\ast \ast}/N). 
\end{align*}
Here we note that 
$$\ell_R(R/I(N))-\ell_R(N^{\ast \ast}/N) \geq \frac{(e-1)(e-2)}{2}$$
by Theorem \ref{4.4}, and $\mu_R(I(N))=\ord(I(N))+1=r$ because $I(N)$ is complete and hence, contracted. 
Therefore, 
$$\ell_R(R/I(M))-\ell_R(F/M) \geq \frac{(e-1)(e-2)}{2}+r-1. $$ 
Since the left hand side is just $\binom{e}{2}$ by Proposition \ref{5.1}, we have the inequality
$$\frac{e(e-1)}{2} \geq \frac{(e-1)(e-2)}{2}+r-1, $$
which implies $e \geq r$. This contradicts to the assumption $e \leq r-1$. Thus, $M$ is indecomposable. 
\end{proof}

When $e=r$, $M=M(I;r)$ is not necessarily indecomposable even if 
$I(M)$ is complete. 

\begin{Example}\label{5.3}
{\rm 
Let $r \geq 4$ and let 
\begin{align*}
    I&=\langle x,y \rangle \langle x,y^2 \rangle \cdots \langle x,y^r \rangle \\
    &=\langle x^{r-i}y^{b_i} \mid 0 \leq i \leq r \rangle
\end{align*} 
where $b_i=\frac{i(i+1)}{2}$. 
Note that $I$ is complete and $I(M(I;e))=I$
for any $2 \leq e \leq r$. By Theorem \ref{5.2}, 
$M(I;2), M(I;3), \dots , M(I;r-1)$ are indecomposable, but $M(I;r)$ is not because 
$$
M(I;r)=\left\langle 
\left(\begin{array}{cc|cccc|cccc}
x&y&y& & & & & & \\ & &x&y& & &y^{c_1}& & \\ & & &x &\ddots& & &\ddots& \\ & & & &\ddots&y& & &\ddots \\ & & & & &x& & & &y^{c_{r-1}}
\end{array}\right)
\right\rangle
$$
has the maximal ideal $\fkm$ as a direct summand. 
}
\end{Example}

However, we can show the following:

\begin{Theorem}\label{5.4}
Let $M=M(I;r)$. Suppose that $I(M)$ is complete. Assume further that $x^{r-1}y \notin I(M)$ and $x^r \in I(M)$. 
Then $M$ is an indecomposable integrally closed $R$-module. 
\end{Theorem}

Before the proof, we make the following lemma. 
Note that in the proof, we use the fact that 
the colength $\ell_R(R/\fka)$ of any monomial ideal $\fka$ in $R$ with respect to $x, y$ 
coincides with the number of monomials in $x, y$ which do not in $\fka$ (see for instance \cite[Lemma 2.5]{GrKi}).

\begin{Lemma}\label{5.5}
Let $\fka$ be an $\fkm$-primary complete monomial ideal with respect to $x, y$, and 
let $r=\ord(\fka) \geq 2$. 
Suppose that $\fka$ is non-simple, and let 
$$\fka=\fkb_1 \fkb_2$$ 
where 
both $\fkb_1$ and $\fkb_2$ are proper complete (monomial) ideals. 
Assume further that $x^r \in \fka$ and $x^{r-1}y \notin \fka$. 
Then we have the following proper inequality:
$$\ell_R(R/\fka)-\ell_R(R/\fkb_1)-\ell_R(R/\fkb_2) > \ord(\fkb_1) \ord(\fkb_2).  $$
\end{Lemma}

\begin{proof}
Let $\fka=\fka_1 \fka_2 \cdots \fka_{\ell}$ be a factorization of $\fka$ into 
simple complete ideals. 
We write 
$$\fka_i=\bar{\langle x^{p_i}, y^{q_i} \rangle}$$ 
where $p_i, q_i$ are coprime positive integers. 
By the assumption that $x^r \in \fka$ and 
$x^{r-1}y \notin \fka$, changing the order of $\fka_i$'s if necessary, we may assume that 
$$\frac{q_{\ell}}{p_{\ell}} \geq \frac{q_{\ell-1}}{p_{\ell-1}} \geq \dots \geq \frac{q_1}{p_1}>1. $$
Then we may also assume that 
$$\fkb_1=\prod_{i \in \Lambda_1} \fka_i, \ \fkb_2=\prod_{i \in \Lambda_2} \fka_i$$
where $\Lambda_1 \cup \Lambda_2=\{1, 2, \dots , \ell \}$, 
$\Lambda_1 \cap \Lambda_2=\emptyset$, $1 \in \Lambda_1$ and 
$\Lambda_2 \neq \emptyset$. Let $r_1=\ord(\fkb_1)$ and $r_2=\ord(\fkb_2)$. Since $x^r \in \fka$, it is easy to see that 
$r_1=\sum_{i \in \Lambda_1} p_i$ and $r_2=\sum_{i \in \Lambda_2} p_i$. 

We first consider the case where $\Lambda_1=\{ 1 \}$, i.e.
$\fkb_1=\fka_1$ and $\fkb_2=\fka_2 \fka_3 \cdots \fka_{\ell}$. 
Then
\begin{align*}
    \ell_R(R/\fka)-\ell_R(R/\fkb_1)-\ell_R(R/\fkb_2)&=r_2 q_1 \\
    &> r_2 p_1 \\
    &=r_1 r_2. 
\end{align*}
In particular, the assertion holds true when $\ell=2$. 

Suppose that $\ell \geq 3$. We may assume that the number of elements of $\Lambda_1$ is at least $2$. Let $\Lambda_1'=\Lambda_1 \setminus \{ 1 \}$, and let $\fkb_1'=\prod_{i \in \Lambda_1'} \fka_i$ and 
$r_1'=\ord(\fkb_1')$. Note that $r_1'=\sum_{i \in \Lambda_1'}p_i=r_1-p_1$. 
We look at the following diagram:
$$
\xymatrix@!=20pt{
&&R \ar@{-}[dl] \ar@{-}[dr] & \\
&\fkb_1' \ar@{-}[dl] \ar@{-}[dr]&&\fkb_2 \ar@{-}[dl] \\
\fka_1 \fkb_1' \ar@{-}[dr]&&\fkb_1' \fkb_2 \ar@{-}[dl] & \\
&\fka_1 \fkb_1' \fkb_2&& \\
}
$$
Note that $\fka=\fka_1 \fkb_1' \fkb_2$ and  $\fkb_1=\fka_1 \fkb_1'$. Then 
\begin{align*}
&    \ell_R(R/\fka)-\ell_R(R/\fkb_1)-\ell_R(R/\fkb_2) \\
&=\ell_R(R/\fkb_1' \fkb_2)+\ell_R(\fkb_1' \fkb_2/\fka_1 \fkb_1' \fkb_2)-\ell_R(R/\fkb_1')-\ell_R(\fkb_1'/\fka_1 \fkb_1')-\ell_R(R/\fkb_2) \\
&>r_1'r_2+\ell_R(\fkb_1' \fkb_2/\fka_1 \fkb_1' \fkb_2)-
\ell_R(\fkb_1'/\fka_1 \fkb_1') \quad \text{by induction} \\
&=r_1'r_2+\ell_R(R/\fka_1 \fkb_1' \fkb_2)-\ell_R(R/\fkb_1' \fkb_2)-\left( \ell_R(R/\fka_1 \fkb_1')-\ell_R(R/\fkb_1')\right) \\
&>r_1'r_2+(r_1'+r_2)p_1+\ell_R(R/\fka_1)-\left(r_1'p_1+\ell_R(R/\fka)\right) \quad \text{by the case where} \ \Lambda_1=\{1\}\\
&=r_1'r_2+r_2p_1 \\
&=r_1 r_2. 
\end{align*}
This completes the proof. 
\end{proof}

Let me give a proof of Theorem \ref{5.4}. 

\begin{proof}[Proof of Theorem \ref{5.4}]
Let $M=M(I;r)$. By Theorem \ref{3.7}, 
$M$ is integrally closed. 
Suppose that $M$ is not indecomposable $R$-module. 
We write $M \cong L_1 \oplus L_2$ for some non-zero integrally closed $R$-modules $L_1$ and $L_2$. Let $e_1=\rank_R(L_1)$ and $e_2=\rank_R(L_2)$. Note that 
$I(L_1)$ and $I(L_2)$ are 
complete and  
$$I(M)=I(L_1)I(L_2). $$ 
Since 
$\ord(I(M))=r=e=e_1+e_2$, 
it follows that 
$\ord(I(L_1))=e_1$ and $\ord(I(L_2))=e_2$. 
Hence, by Lemma \ref{5.5}, 
\begin{equation}\label{ineq1}
    \ell_R(R/I(M))-\ell_R(R/I(L_1))-\ell_R(R/I(L_2))>e_1 e_2. 
\end{equation}
By Theorem \ref{4.4}, 
\begin{equation}\label{ineq2}
\left\{
\begin{array}{c}
    \ell_R(R/I(L_1))-\ell_R(L_1^{\ast \ast}/L_1) \geq  \frac{e_1(e_1-1)}{2} \\
    \ell_R(R/I(L_2))-\ell_R(L_2^{\ast \ast}/L_2) \geq \frac{e_2(e_2-1)}{2} \\
\end{array}
\right. 
\end{equation}
Proposition \ref{5.1} and the above inequalities imply that 
\begin{align*}
    \binom{e}{2}&=\ell_R(R/I(M))-\ell_R(F/M) \\
    &=\ell_R(R/I(M))-\ell_R(L_1^{\ast \ast}/L_1)-\ell_R(L_2^{\ast \ast}/L_2) \\
    &> e_1 e_2+\ell_R(R/I(L_1))+\ell_R(R/I(L_2))-\ell_R(L_1^{\ast \ast}/L_1)-\ell_R(L_2^{\ast \ast}/L_2)  \quad \text{by (\ref{ineq1})} \\
    &\geq e_1 e_2+\frac{e_1(e_1-1)}{2}+\frac{e_2(e_2-1)}{2} \quad \text{by (\ref{ineq2})}.
\end{align*}
Thus, we have 
$$\frac{e(e-1)}{2}-\frac{e_1(e_1-1)}{2}-\frac{e_2(e_2-1)}{2}>e_1e_2. $$
The left hand side is just $e_1e_2$ which is a contradiction. 
Hence $M$ is indecomposable. 
\end{proof}

\section{Consequences}

As a consequence, 
we have the following which immediately implies Theorem \ref{main}. 

\begin{Theorem}\label{6.1}
Let $I$ be a monomial ideal in $($\ref{fixedmono}$)$ 
and assume that $I$ is complete. Let $2 \leq e \leq r$ 
be an integer satisfying $a_{r-e+2}=e-2$. If either
\begin{enumerate}
    \item $r \geq e+1$, or 
    \item $r=e$, $x^{r-1}y \notin I$ and $x^r \in I$
\end{enumerate}
is satisfied, then the associated module $M=M(I;e)$ in Definition \ref{3.1} is an indecomposable integrally closed $R$-module of rank $e$ with $I(M)=I$. 
\end{Theorem}

\begin{proof}
Let $M=M(I;e)$. Since $I$ is complete and $a_{r-e+2}=e-2$, $I=I(M)$ by Proposition \ref{3.3}, and $M$ is integrally closed by Theorem \ref{3.7} (see Remark \ref{3.9}). Note that $r=\ord(I)$ by Lemma \ref{lem}(4). If $r \geq e+1$, 
then $M$ is indecomposable by Theorem \ref{5.2}. 
Suppose that $r=e, x^{r-1}y \notin I$ and $x^r \in I$. Then $M$ is indecomposable by Theorem \ref{5.4}. 
This completes the proof. 
\end{proof}

The condition $a_{r-e+2}=e-2$ is clearly satisfied when $e=2$. Hence we have the following which implies our previous result Theorem \ref{prev} (\cite[Theorem 1.3]{Ha}). 
We note that the proof is significantly simplified than the one in \cite{Ha}. 

\begin{Corollary}\label{6.2}
Let $I$ be a monomial ideal in $($\ref{fixedmono}$)$ and assume that $I$ is complete. If either 
\begin{enumerate}
    \item $r \geq 3$, or  
    \item $r=2$ and $xy \notin I$
\end{enumerate}
is satisfied, then the associated module 
$M=M(I;2)$ is an indecomposable integrally closed $R$-module of rank $2$ with $I(M)=I$.
\end{Corollary}

The condition $a_{r-e+2}=e-2$ is also satisfied if 
$e=3$ and $I$ is complete (see Lemma \ref{lem}(5)). Hence we have the following:

\begin{Corollary}\label{6.3}
Let $I$ be a monomial ideal in $($\ref{fixedmono}$)$ and assume that $I$ is complete. If either 
\begin{enumerate}
    \item $r \geq 4$, or  
    \item $r=3$ and $x^2y \notin I, x^3 \in I$ 
\end{enumerate}
is satisfied, then the associated module 
$M=M(I;3)$ is an indecomposable integrally closed $R$-module of rank $3$ with $I(M)=I$.
\end{Corollary}

Moreover, 
the condition $a_{r-e+2}=e-2$ is also satisfied for any $2 \leq e \leq r$ 
if 
$I$ is simple and complete. Hence we have the following which implies 
the original result of Kodiyalam in \cite {Ko} (see the last assertion of Theorem \ref{kodi}).

\begin{Corollary}\label{6.4}
Let $I$ be a simple complete monomial ideal in $R$ of order $r$. Then for any $2 \leq e \leq r$, the associated module 
$M=M(I;e)$ is an indecomposable integrally closed $R$-module of rank $e$ with $I(M)=I$. 
\end{Corollary}
\begin{proof}
We may assume that $I=\bar{\langle x^r, y^q \rangle}$
where $r, q$ are coprime positive integers with $r < q$. 
Then, by the description (\ref{newton}) of integral closure of monomial ideals, 
it is easy to see that $x^{r-1}y \notin I$ and $x^r \in I$. Thus, we have the assertion by Theorem \ref{6.1}. 
\end{proof}

These corollaries seem to suggest the existence of such modules of rank $e \geq 4$ even in the case where 
a given complete monomial ideal $I$ in (\ref{fixedmono}) does not satisfy the condition $a_{r-e+2}=e-2$. In general, one can ask the following:

\begin{Question}
{\rm 
For a given complete (monomial) ideal $I$ in $R$ of order $r$ and an integer $2 \leq e \leq r$, can we (explicitly) 
construct an indecomposable integrally closed $R$-module $M$ of rank $e$ with $I(M)=I$?
}
\end{Question}

Finally, we illustrate Theorem \ref{main} (Theorem \ref{6.1}) in the following examples.  

\begin{Example}\label{6.6}
{\rm 
Let 
\begin{align*}
    I&=\langle x^7, x^4y, x^3y^2, x^2y^4, xy^5, y^9 \rangle \\
    &=\langle x^3, y\rangle \langle x,y \rangle \bar{\langle x^2,y^3\rangle} \langle x,y^4 \rangle
\end{align*}
be a complete monomial ideal of order $5$. 
Then the following modules are indecomposable integrally closed $R$-modules with the ideal $I$. 

\begin{align*}
    M(I;2)&=\left\langle \left( 
    \begin{array}{ccccc|c|c}
    x^6&x^3y&x^2y^2&xy^4&y^5& y&  \\ & & & & &x&y^8
    \end{array}
    \right) \right\rangle \\
    M(I;3)&=\left\langle \left( 
    \begin{array}{cccc|cc|cc}
    x^5&x^2y&xy^2&y^4&y& & & \\ & & & &x &y&y^4& \\ & & & & &x& &y^7
    \end{array}
    \right) \right\rangle \\
    M(I;4)&=\left\langle \left( 
    \begin{array}{ccc|ccc|ccc}
    x^4&xy&y^2&y& &  & & & \\ & & &x&y& &y^3& & \\ & & & &x&y& &y^3& \\ & & & & &x& & &y^6
    \end{array}
    \right) \right\rangle 
\end{align*}
In this example, the next integrally closed module $M(I;5)$ is 
not indecomposable because 
\begin{align*}
    M(I;5)&=\left\langle \left( 
    \begin{array}{cc|cccc|cccc}
    x^3&y&y& & &  &  & & & \\ & &x&y& &   &  y& & & \\ & & &x& y& & & y^2& & \\ & & & &x&y& & & y^2& \\ & & & & &x& & & & y^5
    \end{array}
    \right) \right\rangle
\end{align*}
has the ideal $\langle x^3, y \rangle$ as a direct summand. 
}
\end{Example}

\begin{Example}\label{6.7}
{\rm 
Let 
\begin{align*}
    I&=\langle x^5, x^4y^2, x^3y^3, x^2y^4, xy^6, y^9 \rangle \\
    &=\bar{\langle x^3, y^4 \rangle} \langle x,y^2 \rangle \langle x,y^3 \rangle
\end{align*}
be a complete monomial ideal of order $5$. 
Then the following modules are indecomposable integrally closed $R$-modules with the ideal $I$. 

\begin{align*}
    M(I;2)&=\left\langle \left( 
    \begin{array}{ccccc|c|c}
    x^4&x^3y^2&x^2y^3&xy^4&y^6& y&  \\ & & & & &x&y^8
    \end{array}
    \right) \right\rangle \\
    M(I;3)&=\left\langle \left( 
    \begin{array}{cccc|cc|cc}
    x^3&x^2y^2&xy^3&y^4&y& & & \\ & & & &x &y&y^5& \\ & & & & &x& &y^7
    \end{array}
    \right) \right\rangle \\
    M(I;4)&=\left\langle \left( 
    \begin{array}{ccc|ccc|ccc}
    x^2&xy^2&y^3&y& &  & & & \\ & & &x&y& &y^3& & \\ & & & &x&y& &y^4& \\ & & & & &x& & &y^6
    \end{array}
    \right) \right\rangle \\
    M(I;5)&=\left\langle \left( 
    \begin{array}{cc|cccc|cccc}
    x&y^2&y& & &  &  & & & \\ & &x&y& &   &  y^2& & & \\ & & &x& y& & & y^2& & \\ & & & &x&y& & & y^3& \\ & & & & &x& & & & y^5
    \end{array}
    \right) \right\rangle
\end{align*}
}
\end{Example}

In this way, we can get many concrete examples of indecomposable integrally closed modules over a two-dimensional regular local ring whose ideal is not necessarily simple.

\section*{Acknowledgments}
This work was supported by JSPS KAKENHI Grant Number JP20K03535.



\end{document}